\documentclass[14pt]{article}
\usepackage{amscd}
\usepackage{amsfonts}
\usepackage{amsmath}
\usepackage{amssymb}
\usepackage{amsthm}
\usepackage{bbm}
\usepackage{CJK}
\usepackage{fancyhdr}
\usepackage{graphicx}
\usepackage{hyperref}
\usepackage{indentfirst}
\usepackage{latexsym}
\usepackage{mathrsfs}
\usepackage{xypic}
\usepackage[all]{xy}           %xypic macro for latex2.09
\usepackage{amsfonts,latexsym}
\usepackage{xspace}
\usepackage{epsfig}
\usepackage{float}
\usepackage{txfonts}
\usepackage{amssymb}
\usepackage{graphicx}
\usepackage{amsmath}
\usepackage{xcolor}
\usepackage[top=1in,bottom=1in,left=1.25in,right=1.25in]{geometry}
\usepackage{epsf,graphicx,epsfig}
\usepackage{amsmath,latexsym,amssymb,amsthm}
\usepackage[nospace,noadjust]{cite}
\usepackage{textcomp}
\usepackage{setspace,cite}
\usepackage{stmaryrd}
\usepackage{tikz-cd} 
\usepackage{tikz}
\usepackage{cancel}
\usetikzlibrary{calc,shapes,arrows,decorations.markings,decorations.pathmorphing}
\usepackage{color}
\usepackage{url}
\usepackage{enumerate}
\usepackage[mathscr]{euscript}
\swapnumbers
    \newtheorem{thm}{Theorem}[section]
    \newtheorem{Prop}[thm]{Proposition}
    
    \newtheorem*{Proof*}{Proof}
     
   \newtheorem{Lemma}[thm]{Lemma}
    \newtheorem{Thm}[thm]{Theorem}

    \newtheorem{subsec}[thm]{}
\theoremstyle{Definition}
    \newtheorem{Def}[thm]{Definition}
        \newtheorem{Rem}[thm]{Remark}
    \newtheorem{Exam}[thm]{Example}
    
\newtheorem{Coro}[thm]{Corollary}
\theoremstyle{remark}

\usepackage{amssymb}

\usepackage{hyperref}
\hypersetup{
	colorlinks,
	citecolor=blue,
	filecolor=black,
	linkcolor=blue,
	urlcolor=black
}
\tikzset{
  curve/.style={
    settings={#1},
    to path={
      (\tikztostart)
      .. controls ($(\tikztostart)!\pv{pos}!(\tikztotarget)!\pv{height}!270:(\tikztotarget)$)
      and ($(\tikztostart)!1-\pv{pos}!(\tikztotarget)!\pv{height}!270:(\tikztotarget)$)
      .. (\tikztotarget)\tikztonodes
    },
  },
  settings/.code={%
    \tikzset{quiver/.cd,#1}%
    \def\pv##1{\pgfkeysvalueof{/tikz/quiver/##1}}%
  },
  quiver/.cd,
  pos/.initial=0.35,
  height/.initial=0,
}

\date{}
\begin{document}
\renewcommand{\baselinestretch}{1.2}
\renewcommand{\arraystretch}{1.0}
\title{\bf On Affine Version of Hom-Lie Algebras}
%\title{\bf Equivariant cohomology of Lie algebras using Green functors}
\date{}
\author{{\bf Tarik Anowar$^{1}$,~Ripan Saha$^{2}$\footnote
        { Corresponding author:~~Email: ripanjumaths@gmail.com}}\\
{\small 1.  Department of Mathematics, Raiganj University, Raiganj 733134, West Bengal, India}\\
{\small 3. Department of Mathematics, Raiganj University, Raiganj 733134, West Bengal, India}
}
 \maketitle
\begin{center}
\begin{minipage}{12.cm}
\begin{center}{\bf ABSTRACT}\end{center}

This paper introduces Hom-type analogues of affine algebraic structures, termed Hom-affgebras. 
Extending Brzezi\'nski's theory of affgebras and the Hom-algebra framework developed by Hartwig--Larsson--Silvestrov, we define and study Hom-associative, Hom-pre-Lie, and Hom-Lie affgebras, 
where the classical identities are twisted by an affine self-map. We show how Hom-associative, Hom-pre-Lie, and Hom-Lie affgebras are related to one another. The main focus of this paper is on Hom-Lie affgebras and their fibers. We study the concept of generalized derivations for Hom-Lie algebras, extending the notion of generalized derivations for Lie algebras. We explore the close relationship between Hom-Lie affgebras and such derivations. We show that every Hom-Lie affgebra both determines and is determined by a Hom-Lie algebra together with such a generalized derivation and a constant. Furthermore, we establish that a homomorphism between Lie affgebras corresponds to a homomorphism between their associated Lie fibers along with a constant, and vice versa.

\medskip

\medskip

\end{minipage}
\end{center}
\normalsize\vskip0.5cm

\section{Introduction}
The notion of \emph{affgebra} was introduced as a natural extension of affine geometry into the realm of algebraic structures, replacing linear dependence on a base point with affine and heap-theoretic operations. The origin of this idea can be traced back to the study of \emph{heaps} (or \emph{torsors})~\cite{Pru1, Bae1, Brz22}, structures with a ternary operation satisfying the Mal’tsev and associativity identities, which capture the essence of groups but without a fixed identity element. In recent years, Brzeziński and his collaborators \cite{Brz1, Brz2, Brz3, Brz222, And1, And2} have redefined a systematic theory of affine spaces using heaps, and developed the theory of affgebras as algebraic structures defined on affine spaces instead of vector spaces, providing an intrinsic, basepoint-free approach to algebraic geometry. Historically, the study of Lie affgebras started in \cite{Gra1}, exploring vector-space-valued Lie brackets on affine spaces. This theory was later developed for the differential geometry of affine-fibre bundles (\emph{AV-geometry})\cite{Gra4, Gra1} and for a frame-independent approach to Lagrangian mechanics~\cite{Gra3, Tul1}. Reference~\cite{Brz1} then proposed a new, more general framework for Lie affgebras that does not rely on or assume an underlying vector space. This framework revealed elegant analogues of associative, Lie, and pre-Lie algebras within the affine category, highlighting the deep interplay between algebraic operations and affine geometry. In \cite{And2}, the authors studied Lie affgebras and their relationship with the Lie algebras.

In recent years, \emph{Hom-type algebras} emerged as generalizations of classical algebraic structures. Initiated by Hartwig, Larsson, and Silvestrov \cite{HLS} by introducing Hom-Lie algebras in the context of deformations of the Witt and Virasoro algebras, Hom-algebras are defined by twisting their defining identities via a linear self-map. This twisting mechanism, replacing strict associativity or Jacobi conditions with Hom-associativity and Hom-Jacobi conditions, allows flexible deformation theories and connections with quantum groups and $\sigma$-derivations. The resulting Hom-associative \cite{MS10,CS} and Hom-Lie algebras\cite{HLS, MS10} have since inspired extensive research on Hom-pre-Lie \cite{MS, GS23}, Hom-Leibniz \cite{MS08,MS, saha20, saha21}, and other generalized algebraic structures. Recently, we have studied hom version of heaps and trusses in \cite{AST}.

The main goal of this paper is to unify these two directions of research—namely, the affinization of algebraic structures and Hom-type algebras—by developing the theory of \emph{Hom-affgebras}. Specifically, we introduce and study \emph{Hom-associative affgebras} and \emph{Hom-Lie affgebras} as Hom-type analogues of the associative and Lie affgebras of Brzezi\'nski and his collaborators. This approach provides a unified framework connecting affine and Hom-type algebraic theories and reveals how Hom-affine structures naturally retract to corresponding Hom-algebras on tangent spaces at chosen base points and vice versa.

An additional focus of this work is the development of \emph{generalized derivations} for Hom-Lie algebras. The notion of generalized derivations was originally introduced in the context of Lie algebras by Leger and Luks \cite{Leg1} as a way to extend the concept of derivations to include centroids and quasi-centroids, thereby encompassing a wider class of linear endomorphisms satisfying generalized Leibniz-type relations. In the present work, we construct the Hom-analogues of these generalized derivations, where the compatibility conditions are twisted by the Hom-map. This generalization allows for the investigation of Hom-Lie affgebra structures arising from generalized derivations for Hom-Lie algebras, extending the work of Andruszkiewicz, Brzezi\'nski, and Radziszewski in the affine framework \cite{And2} to the Hom-setting. 

The paper is organized as follows: Section 2 recalls the fundamental notions of heaps, affine modules, affine morphisms, affgebras, and Lie affgebras providing examples and the connection with classical affine spaces and vector spaces. In Section 3, we define Hom-associative affgebra, Hom-Lie affgebra and some interesting results. Section 4 explores the interrelationships between various Hom-affgebraic structures, showing how Hom-associative affgebras yield Hom-Lie and Hom-pre-Lie affgebras, thus extending classical correspondences known in non-Hom settings. In Section 5, we introduce and investigate the concept of \emph{generalized derivations} within the framework of Hom--Lie algebras. Building upon this foundation, we develop a systematic correspondence between Hom--Lie affgebras and such derivations. In particular, we demonstrate that every Hom--Lie affgebra can be characterized by a Hom--Lie algebra endowed with a Hom-version of generalized derivation and an additional constant element, and conversely, each such triple gives rise to a Hom--Lie affgebra. Moreover, we establish a structural correspondence between morphisms of Hom--Lie affgebras and morphisms of their fiber Hom--Lie algebras: a homomorphism between Hom--Lie affgebras is equivalent to a homomorphism between the associated Hom--Lie fibers together with a compatible constant and satisfying some compatibility conditions, and the converse also holds.

This study thus lays the foundation for a comprehensive theory of Hom-affgebras and their derivations, bridging affine and Hom-type algebraic geometry. It opens up further avenues toward Hom-analogues of Leibniz algebras, dendriform algebras, and other types of non-associative algebras, as well as their cohomological structures, within the affine and heap-theoretic context.

\section{Preliminaries}

In this section, we recall the fundamental notions of heaps, affine spaces and affgebras which will be used throughout the paper. For further details, we refer the reader to \cite{Bre1, Brz1, Brz1, Brz2, Brz3, Brz22, Brz222, And1, And2, Jac1}. Throughout the paper, $\mathbb{K}$ denotes a field of characteristic not equal to 2, unless stated otherwise.

\begin{Def}\cite{Pru1,Brz22,Brz222}
A \emph{heap} is a set $\mathcal{A}$ equipped with a ternary operation 
\[
\langle -,-,- \rangle : \mathcal{A} \times \mathcal{A} \times \mathcal{A} \longrightarrow \mathcal{A}, \quad (a,b,c) \mapsto \langle a,b,c \rangle,
\]
satisfying, for all $a,b,c,d,e \in \mathcal{A}$:
\begin{enumerate}
\item $\langle \langle a,b,c \rangle, d, e \rangle = \langle a,b, \langle c,d,e \rangle \rangle$ \quad (associativity),
\item $\langle a,a,b \rangle = b = \langle b,a,a \rangle$ \quad (Mal'tcev identity).
\end{enumerate}
\end{Def}
A heap is called \emph{abelian} if $\langle a,b,c\rangle = \langle c,b,a\rangle$ for all $a,b,c \in \mathcal{A}$.
\begin{Rem}\label{re1}
Let $(\mathcal{A},\bullet)$ be an abelian group. Then $\mathcal{A}$ becomes an abelian heap with the ternary operation 
\[
\langle a,b,c\rangle = a\bullet b^{-1}\bullet c, \quad \text{for all } a,b,c \in \mathcal{A}.
\]
On other hand, for any non-empty heap $\mathcal{A}$ and an element $o \in \mathcal{A}$, the set $\mathcal{A}$ endowed with the binary operation  
\[
a +_{o} b = \langle a, o, b \rangle
\]  
is a group (abelian if $\mathcal{A}$ is abelian), known as the \emph{retract} of $\mathcal{A}$ at $o$, denoted by $G(\mathcal{A},o)$.  
Moreover, the inverse of $a \in \mathcal{A}$ in $G(\mathcal{A},o)$ is given by  
\[
a^{-1} = \langle o, a, o \rangle.
\]
\end{Rem}

A heap morphism from $(\mathcal{A},\langle -,-,- \rangle_{1})$ to $(\mathcal{B},\langle -,-,- \rangle_{2})$ is a map $f:\mathcal{A}\to \mathcal{B}$ such that $f(\langle a,b,c\rangle_{1})=\langle f(a),f(b),f(c)\rangle_{2}$ for all $a,b,c\in \mathcal{A}$.
The axioms of an abelian heap imply that the placement of brackets 
between any odd number of elements of $\mathcal{A}$ does not affect the value of the expression. 
Hence, for any odd number of elements we write
\[
  \langle a_1,a_2,\ldots,a_{2n+1}\rangle \in \mathcal{A}
\]
to denote any possible repeated application of the ternary operation $\langle-,-,-\rangle$ 
to the tuple $(a_1,a_2,\ldots,a_{2n+1})$.

\medskip

Moreover, any pair of identical neighbouring elements cancel out; we call this 
the \emph{cancellation rule}. In addition, we have the \emph{reshuffling rule}:
\[
  \langle a_1,a_2,\ldots,a_{2n+1}\rangle
  = \langle a_{\sigma(1)},a_{\tau(2)},\ldots,a_{\tau(2n)},a_{\sigma(2n+1)}\rangle,
\]
where $\sigma$ is any permutation of the odd indices 
$\{1,3,\ldots,2n+1\}$ and $\tau$ is any permutation of the even indices 
$\{2,4,\ldots,2n\}$. 

\medskip

As a consequence, any pair of identical elements occurring in positions 
of opposite parity can be removed, provided the remaining elements are 
reshuffled so that each one keeps the same parity as its original position. 
The value of the entire bracketed expression remains unchanged.
\begin{Def}\cite{Brz1}\label{Def2.3}  
An \emph{affine space} over the field $\mathbb{K}$ is a non-empty abelian heap $\mathcal{A}$ equipped with a ternary action
\[
\blacktriangleright : \mathbb{K} \times \mathcal{A} \times \mathcal{A} \longrightarrow \mathcal{A}, 
\qquad
(\alpha, a, b) \longmapsto \alpha \blacktriangleright_a b,
\]
satisfying the following axioms for all $a,b,c \in \mathcal{A}$ and $\xi,\zeta \in \mathbb{K}$:

\begin{itemize}
    \item[(a)] The maps 
    \[
    \xi \blacktriangleright_a - : \mathcal{A} \to \mathcal{A} 
    \quad\text{and}\quad
    - \blacktriangleright_a b : \mathbb{K} \to \mathcal{A}
    \]
    are heap homomorphisms, where $\mathbb{K}$ is regarded as a heap with the operation 
    \[
    (\xi,\zeta,\gamma) \longmapsto \xi - \zeta + \gamma.
    \]

    \item[(b)] Compatibility with scalar multiplication:
    \[
    (\xi \zeta) \blacktriangleright_a b = 
    \xi \blacktriangleright_a \!\bigl(\zeta \blacktriangleright_a b\bigr).
    \]

    \item[(c)] Independence of the choice of base:
    \[
    \xi \blacktriangleright_a b 
    = 
    \bigl\langle 
      \xi \blacktriangleright_c b,\,
      \xi \blacktriangleright_c a,\,
      a
    \bigr\rangle.
    \]

    \item[(d)] Action of $0$ and $1$:
    \[
    0 \blacktriangleright_a b = a,
    \qquad 
    1 \blacktriangleright_a b = b.
    \]
\end{itemize}
\end{Def}
Here, the element $a$ in $\xi \blacktriangleright_a b$ is called the \emph{base} of the action.  

A homomorphism of affine space is a heap homomorphism $f$ satisfying
\[
f(\xi \blacktriangleright_a b) 
= 
\xi \blacktriangleright_{f(a)} f(b).
\]

When $\mathbb{K}$ is a commutative ring with identity, we typically use the term \emph{affine module} rather than affine space.
\begin{Exam}\label{Exam1}
Let
\[
sna(n,\mathbb{R}) := \Bigl\{ A \in \mathfrak{sl}(n+1,\mathbb{R}) \ \Big|\ 
\forall j,\ \sum_{i=1}^{n} a_{ij} = 1 = \sum_{i=1}^{n} a_{ji} \Bigr\}.
\]

Then $sna(n,\mathbb{R})$ is an affine space structure with the ternary operation and action
\[
\langle A,B,C \rangle = A - B + C, \qquad
\lambda \blacktriangleright_{A} B = \lambda B + (1 - \lambda) A,
\]
where addition and scalar multiplication are the usual operations on matrices. For details proof see the Proposition 4.1 in paper \cite{Brz2}.
\end{Exam}
Let us briefly explain that the definition of an affine space given here is equivalent 
to the classical definition involving a set $\mathcal{A}$ with a free and transitive action $+$
of a $\mathbb{K}$-vector space $\overrightarrow{\mathcal{A}}$. In particular, for all $b,c \in \overrightarrow{\mathcal{A}}$, there exists a 
unique vector from $b$ to $c$, and every vector in $\overrightarrow{\mathcal{A}}$ can be obtained in this way. 
Moreover, any point $a \in \mathcal{A}$ can be uniquely translated to another point 
$a + \overrightarrow{bc} \in \mathcal{A}$. This combination of three points equips $\mathcal{A}$ with the 
structure of an abelian heap:
\[
\langle a, b, c \rangle := a + \overrightarrow{bc}.
\]
Thus, $\mathcal{A}$ becomes an affine $\mathbb{K}$-space in the sense of Definition (\ref{Def2.3}), with the action
\[
\xi \blacktriangleright_{a} b := a + \xi \,\overrightarrow{ab}.
\]
Conversely, starting from an affine $\mathbb{K}$-space $\mathcal{A}$ as in Definition \ref{Def2.3}, one can reconstruct the underlying vector space by fixing any $o \in \mathcal{A}$ and defining vector addition and scalar multiplication as
\begin{equation}\label{eq1}
a + b := \langle a, o, b \rangle, \qquad 
\alpha a := \alpha \blacktriangleright_{o} a.
\end{equation}

The resulting vector space (or more generally, a $\mathbb{K}$-module if $\mathbb{K}$ is a commutative ring rather than a field), we call it the tangent space to $\mathcal{A}$ at $o$ or the vector space fibre of $\mathcal{A}$ at $o$ and is denoted by $T_{o}\mathcal{A}$. In this notation, 
$\overrightarrow{bc} = \langle o, b, c \rangle \in T_{o}\mathcal{A}$.

\medskip
\noindent
\begin{Rem}\cite{Brz1}\label{Rem1}
Different choices of the reference point $o \in \mathcal{A}$ lead to 
isomorphic $\mathbb{K}$- vector space. More precisely, for $o,e \in \mathcal{A}$, define
\[
\tau^{e}_{o} \colon T_{o}\mathcal{A} \longrightarrow T_{e}\mathcal{A}, \qquad 
a \longmapsto \langle a, o, e \rangle.
\]
This is an automorphism of heaps and an isomorphism of abelian groups 
$T_{o}\mathcal{A} \cong T_{e}\mathcal{A}$, called the \textit{translation isomorphism} 
(the inverse is $\tau^{o}_{e}$). Moreover, for all $a \in \mathcal{A}$ and $\xi \in \mathbb{K}$,
\[
\xi \blacktriangleright_{e} \tau^{e}_{o}(a) 
= \langle \xi \blacktriangleright_{e} a, \xi \blacktriangleright_{e} o, \xi \blacktriangleright_{e} e \rangle
= \langle \xi \blacktriangleright_{e} a, \xi \blacktriangleright_{e} o, o, e \rangle
= \langle \xi \blacktriangleright_{o} a, o, e \rangle 
= \tau^{e}_{o} (\xi \blacktriangleright_{o} a).
\]

The first equality holds because $\xi \blacktriangleright_{e}-$ is a heap homomorphism. The second follows from the Mal'cev identity and $\xi\blacktriangleright_{e}e=e$, while the penultimate equality 
uses the base-exchange property in Definition \ref{Def2.3}. This proves that 
$\tau^{e}_{o}$ is a $\mathbb{K}$- vector space isomorphism $T_{o}\mathcal{A} \cong T_{e}\mathcal{A}$.

Traditionally, an affine map is defined as a pair consisting of a function 
$f \colon \mathcal{A} \to \mathcal{B}$ and a uniquely determined linear map 
$\overrightarrow{f} \colon \overrightarrow{\mathcal{A}} \to \overrightarrow{\mathcal{B}}$, called the \textit{linearisation} of $f$, such that
\[
\overrightarrow{f}\bigl(\overrightarrow{ab}\bigr) = \overrightarrow{f(a)f(b)}.
\]
Equivalently, $f$ is a homomorphism of affine spaces in the sense of Definition (\ref{Def2.3}). 
Fixing $o_{\mathcal{A}} \in \mathcal{A}$ and $o_{\mathcal{B}} \in \mathcal{B}$ and defining the vector space structures 
$T_{{o}_{\mathcal{A}}}\mathcal{A}$ and $T_{{o}_{\mathcal{B}}}\mathcal{B}$ as in (\ref{eq1}), the linearisation of $f \colon \mathcal{A} \to \mathcal{B}$ is given by
\begin{equation}\label{eq1a}
\overrightarrow{f}\colon T_{o_{\mathcal{A}}}\mathcal{A} \longrightarrow T_{o_{\mathcal{B}}}\mathcal{B}, \qquad 
a \longmapsto \langle f(a), f(o_{\mathcal{A}}), o_{\mathcal{B}} \rangle.
\end{equation}
That is; for every $f\in $ Aff $(\mathcal{A},\mathcal{B})$ induces the unique linear transformation $\overrightarrow{f}\colon T_{o_{\mathcal{A}}}\mathcal{A}_{o_{\mathcal{A}}} \longrightarrow T_{o_{\mathcal{B}}}\mathcal{B}$, by $\overrightarrow{f}:a\mapsto f(a)-f(o_{\mathcal{A}})$. Conversely, given any linear transformation $\overrightarrow{f}\colon T_{o_{\mathcal{A}}}\mathcal{A} \longrightarrow T_{o_{\mathcal{B}}}\mathcal{B}$ and an any element $b\in \mathcal{B}$, the map $f:a\mapsto \overrightarrow{f}(a)+b$ is an affine transformation.
\end{Rem}
\begin{Def}\cite{Brz1}
 An affgebra ($\mathbb{K}$-affgebra) is an affine $\mathbb{K}$-space $\mathcal{A}$ together with a bi-affine multiplication $\mathcal{A}\times \mathcal{A}\rightarrow\mathcal{A}$.
\end{Def}
If the bi-affine multiplication is satisfy associative property, then it is called an associative affgebra (associative $\mathbb{K}$-affgebra).
A homomorphism of affgebra is an affine map that preserve the bi-affine multiplication.
\begin{Exam}
Let $(sna(n,\mathbb{R}),\langle-,-,-\rangle,\blacktriangleright)$ be an affine space, with the same ternary operation in Example (\ref{Exam1}). Then, $sna(n,\mathbb{R})$ is an associative affgebra with the bi-affine multiplication 
$$\mu(A,B)=AB,$$
where multiplication is usual operation on matrices.
\end{Exam}
\begin{Def}\cite{And2}
Let $\mathcal{A}$ be an affine $\mathbb{K}$-space. A Lie bracket on $\mathcal{A}$ is a bi-affine map $\{-,-\}:\mathcal{A}\times\mathcal{A}\rightarrow\mathcal{A}$ such that, for all $a,b,c\in \mathcal{A}$,
\begin{enumerate}
\item[(i)] Affine anti-symmetry, that is,
$$
\langle\{a,b\},\{a,a\},\{b,a\}\rangle=\{b,b\};
$$
\item[(ii)] The affine Jacobi identity, that is,
$$
\langle\{a,\{b,c\}\},\{a,\{a,a\}\},\{b,\{c,a\}\},\{b,\{b,b\}\},\{c,\{a,b\}\}\rangle=\{c,\{c,c\}\}.
$$
\end{enumerate}
\end{Def}
A homomorphism of Lie affgebra is an affine map that preserve Lie bracket.
\begin{Exam}
Let $(sna(n,\mathbb{R}),\langle-,-,-\rangle,\blacktriangleright)$ be an affine space, with the same ternary operation in Example (\ref{Exam1}). Then, $sna(n,\mathbb{R})$ is an Lie affgebra with the bi-affine bracket 
$$\{A,B\}=AB-BA+B,$$
where addition and multiplication are usual operation on matrices.
\end{Exam}
\section{Hom-Lie affgebra}
In this section we define Hom-associative affgebra, Hom-Lie affgebra and some interesting results.
\begin{Def}
 Let $\mathcal{A}$ be an affine space. A Hom-associative affgebra is a triple of $(\mathcal{A},\mu,\alpha)$ with the bi-affine map $\mu:\mathcal{A}\times\mathcal{A}\rightarrow \mathcal{A}$ and the unary affine map $\alpha:\mathcal{A}\rightarrow \mathcal{A}$ such that for all $a,b,c\in \mathcal{A}$  
$$
\mu(\alpha(a),\mu(b,c))=\mu(\mu(a,b),\alpha(c)).
$$
\end{Def}
If the affine map $\alpha$ is multiplicative, i.e; $\alpha(\mu(a,b))=\mu(\alpha(a),\alpha(b))$, for all $a,b\in \mathcal{A}$, then is called multiplicative Hom-associative affgebra.
\begin{Exam}
Let $(sna(n,\mathbb{R}),\langle-,-,-\rangle,\blacktriangleright)$ be an affine space, with the same ternary operation in Example (\ref{Exam1}). Then, $sna(n,\mathbb{R})$ is a Hom-associative affgebra with unary affine map $\alpha:sna(n,\mathbb{R})\rightarrow sna(n,\mathbb{R})$ and the bi-affine multiplication 
$$\alpha(A)=PAP^{-1},~~~~~~~~~~~~~~\mu(A,B)=P(AB)P^{-1},~~~~~~~~~\text{for P $\in sna(n,\mathbb{R})$ and $det(P)\neq 0$,}$$
where multiplication is usual multiplication on matrices. It is a routine work to show the affine hom-associativity using the associativity of matrices multiplication.
\end{Exam}
\begin{Def}
Let $(\mathcal{A},\mu,\alpha)$ be a Hom-associative affgebra. A affine morphism, $f:\mathcal{A}\rightarrow \mathcal{A}$ is called a morphism of Hom-associative affgebra if 
\begin{enumerate}
\item[(i)] $f\circ\alpha=\alpha\circ f$.
\item[(ii)] For all $a,b\in \mathcal{A}$,~~~~
$
f(\mu(a,b))=\mu(f(a),f(b)).
$
\end{enumerate}
\end{Def}
\begin{Def}
Let $\mathcal{A}$ be an affine space over $\mathbb{K}$. A Hom-Lie affgebra is a triple $(\mathcal{A},\{~,~\},\alpha)$ with the bi-affine map $\{~,~\}:\mathcal{A} \times \mathcal{A}\rightarrow \mathcal{A}$ and the affine map $\alpha:\mathcal{A}\rightarrow \mathcal{A}$. For all $a,b,c\in \mathcal{A}$, following condition holds
\begin{enumerate}
\item[(a)] Affine anti-symmetry, that is,
$$
\langle\{a,b\},\{a,a\},\{b,a\}\rangle=\{b,b\};
$$ 
\item[(b)]Affine Hom-Jacobi identity, that is,
$$
\langle \{\alpha(a),\{b,c\}\},\{\alpha(a),\{a,a\}\},\{\alpha(b),\{c,a\}\},\{\alpha(b),\{b,b\}\},\{\alpha(c),\{a,b\}\}\rangle=\{\alpha(c),\{c,c\}\}.
$$
\end{enumerate}
\end{Def}
If the affine map $\alpha$ is multiplicative, i.e; $\alpha\{a,b\}=\{\alpha(a),\alpha(b)\}$, for all $a,b\in \mathcal{A}$, then is called multiplicative Hom-Lie affgebra.
\begin{Exam}
Let $(sna(n,\mathbb{R}),\langle-,-,-\rangle,\blacktriangleright)$ be an affine space, with the same ternary operation in Example (\ref{Exam1}). Then, $sna(n,\mathbb{R})$ is a Hom-Lie affgebra with unary affine map $\alpha:sna(n,\mathbb{R})\rightarrow sna(n,\mathbb{R})$ and the bi-affine bracket 
$$\alpha(A)=PAP^{-1},~~~~~~~~~~~\{A,B\}=P(AB)P^{-1}-P(BA)P^{-1}+\alpha(B),~~~~~~~~\text{for P $\in sna(n,\mathbb{R})$ and $det(P)\neq 0$.}$$
Similar to the Proposition (\ref{Prop1}), it is easy to show the affine anti-symmetry and affine Hom-Jacobi identity using associativity and distributive property of matrices multiplication.
\end{Exam}
\begin{Exam}
Let $\mathcal{A}$ be an affine space. Let $\phi:\mathcal{A}\rightarrow\mathcal{A}$ and $\alpha:\mathcal{A}\rightarrow\mathcal{A}$ be two endomorphism of affine space $\mathcal{A}$. Define the bracket $\{-,-\}:\mathcal{A}\times\mathcal{A}\rightarrow\mathcal{A}$ by $\{a,b\}=\phi(a)$. The bracket is bi-affine as both the constant map and $\phi$ are affine endomorphisms. For all $a,b,c\in \mathcal{A}$, we have
$$
\langle\{a,b\},\{a,a\},\{b,a\}\rangle=\langle\phi(a),\phi(a),\phi(b)\rangle=\phi(b)=\{b,b\};
$$
and
\begin{align*}
&~~~\langle\{\alpha(a),\{b,c\}\},\{\alpha(a),\{a,a\}\},\{\alpha(b),\{c,a\}\},\{\alpha(b),\{b,b\}\},\{\alpha(c),\{a,b\}\}\rangle\\
&=\langle\{\alpha(a),\phi(b)\},\{\alpha(a),\phi(a)\},\{\alpha(b),\phi(c)\},\{\alpha(b),\phi(b)\},\{\alpha(c),\phi(a)\}\rangle\\
&=\langle\phi(\alpha(a)),\phi(\alpha(a)),\phi(\alpha(b)),\phi(\alpha(b)),\phi(\alpha(c))\rangle\\
&=\phi(\alpha(c))\\
&=\{\alpha(c),\{c,c\}\}.
\end{align*}
Hence, the bracket $\{-,-\}$ satisfies affine anti-symmetry and affine Hom-Jacobi identity. Therefore, $(\mathcal{A},\{-,-\},\alpha)$ is Hom-Lie affgebra.
\end{Exam}
\begin{Prop}
Given an affine space $\mathcal{A}$, $\alpha:\mathcal{A}\rightarrow\mathcal{A}$ be an endomorphism of affine space and $\xi\in \mathbb{K}$, define the bracket
$$
\{-,-\}:\mathcal{A}\times\mathcal{A}\rightarrow\mathcal{A},~~~~~~~~~\{a,b\}=\xi\blacktriangleright_{a}b.
$$
Then, $(\mathcal{A},\{-,-\},\alpha)$ is a Hom-Lie affgebra.
\end{Prop}
\begin{proof}
The bracket $\{~,~\}$ is a bi-affine and satisfy affine anti-symmetry follows from Proposition 3.4 of the paper \cite{Brz1}. Let we fixed an element $o\in A$, and using the base change property of the Definition (\ref{Def2.3}). we can write action $\blacktriangleright_{a}$, in terms of $\blacktriangleright_{o}$ to obatin the follows 
$$
\{\alpha(a),\{b,c\}\}=\langle \xi^{2}\blacktriangleright_{o}c,\xi^{2}\blacktriangleright_{o}b,\xi\blacktriangleright_{o}b,\xi\blacktriangleright_{o}\alpha(a),\alpha(a)\rangle.
$$
Now, we have
\begin{align*}
&~~~\langle\{\alpha(a),\{b,c\}\},\{\alpha(a),\{a,a\}\},\{\alpha(b),\{c,a\}\},\{\alpha(b),\{b,b\}\},\{\alpha(c),\{a,b\}\}\rangle\\
&=\langle \xi^{2}\blacktriangleright_{o}c,\xi^{2}\blacktriangleright_{o}b,\xi\blacktriangleright_{o}b,\xi\blacktriangleright_{o}\alpha(a),\alpha(a), \xi^{2}\blacktriangleright_{o}a,\xi^{2}\blacktriangleright_{o}a,\xi\blacktriangleright_{o}a,\xi\blacktriangleright_{o}\alpha(a),\alpha(a)\\
&~~~~~~ \xi^{2}\blacktriangleright_{o}a,\xi^{2}\blacktriangleright_{o}c,\xi\blacktriangleright_{o}c,\xi\blacktriangleright_{o}\alpha(b),\alpha(b), \xi^{2}\blacktriangleright_{o}b,\xi^{2}\blacktriangleright_{o}b,\xi\blacktriangleright_{o}b,\xi\blacktriangleright_{o}\alpha(b),\alpha(b)\\
&~~~~~~
\xi^{2}\blacktriangleright_{o}b,\xi^{2}\blacktriangleright_{o}a,\xi\blacktriangleright_{o}a,\xi\blacktriangleright_{o}\alpha(c),\alpha(c)\rangle\\
&=\langle\xi\blacktriangleright_{o}c,\xi\blacktriangleright_{o}\alpha(c),\alpha(c)\rangle\\
&=\{\alpha(c),\{c,c\}\}.
\end{align*}
Hence, the bracket is satisfy the affine Hom-Jacobi identity. Therefore, $(\mathcal{A},\{-,-\},\alpha)$ is an Hom-Lie affgebra. 
\end{proof}
\begin{Def}
Let $(\mathcal{A},\{~,~\},\alpha)$ be a Hom-Lie affgebra. An affine morphism, $f:\mathcal{A}\rightarrow \mathcal{A}$ is called a morphism of Hom-Lie affgebra if, 
\begin{enumerate}
\item[(i)] $f\circ\alpha=\alpha\circ f$.
\item[(ii)] For all $a,b\in \mathcal{A}$,~~~~~~~
$
f\{a,b\}=\{f(a),f(b)\}.
$
\end{enumerate}
\end{Def}
\section{Relationship Between Different Hom-affgebras}
Henceforth, we assume Hom-affgebra (associative or non-associative) means it is multiplicative Hom-affgebra.
\begin{Prop}\label{Prop1}
A Hom-associative affgebra $\mathcal{A}$ is a Hom-Lie affgebra with the bracket
$$
\{a,b\}=\langle ab,ba,\alpha(b)\rangle,
$$
for all $a,b\in \mathcal{A}$, where $\alpha$ is a Hom map.
\end{Prop}
\begin{proof}
For any $a,b,c,d\in \mathcal{A}$ and any scalar $\zeta\in \mathbb{K}$, we get
\begin{align*}
\{a,\langle b,c,d\rangle\}&=\langle a\langle b,c,d\rangle,\langle b,c,d\rangle a,\alpha\langle b,c,d\rangle\rangle\\
&=\langle ab,ac,ad,ba,ca,da,\alpha(b),\alpha(c),\alpha(d)\rangle\\
&=\langle ab,ba,\alpha(b),ac,ca,\alpha(c),ad,da,\alpha(d)\rangle\\
&=\langle \{a,b\},\{a,c\},\{a,d\}\rangle.
\end{align*}
Similarly, we have $\{\langle b,c,d\rangle,a\}=\langle\{b,a\},\{c,a\},\{d,a\}\rangle$.
\begin{align*}
\{a,\zeta\blacktriangleright_{b}c\}&=\langle a(\zeta\blacktriangleright_{b}c),(\zeta\blacktriangleright_{b}c)a,\alpha(\zeta\blacktriangleright_{b}c)\rangle\\
&=\langle \zeta\blacktriangleright_{ab}ac,\zeta\blacktriangleright_{ba}ca,\zeta\blacktriangleright_{\alpha(b)}\alpha(c)\rangle\\
&=\langle \zeta\blacktriangleright_{\langle ab,ba,\alpha(b)\rangle}ac,\zeta\blacktriangleright_{\langle ab,ba,\alpha(b)\rangle}ab,ab,\zeta\blacktriangleright_{\langle ab,ba,\alpha(b)}ca,\zeta\blacktriangleright_{\langle ab,ba \alpha(b)\rangle}ba,\\
&~~~~~ba,\zeta\blacktriangleright_{\langle ab,ba,\alpha(b)}\alpha(c),\zeta\blacktriangleright_{\langle ab,ba,\alpha(b)\rangle}\alpha(b),\alpha(b)\rangle~~~~~~~~~~~~~~~~~~~~~\text{\{Base change property \} }\\
&=\langle\zeta\blacktriangleright_{\langle ab,ba,\alpha(b)\rangle}\langle ac,ca,\alpha(c)\rangle,\zeta\blacktriangleright_{\langle ab,ba,\alpha(b)\rangle}\langle ab,ba,\alpha(b)\rangle,\langle ab,ba,\alpha(b)\rangle\rangle\\
&=\zeta\blacktriangleright_{\langle ab,ba,\alpha(b)}\langle ac,ca,\alpha(c)\rangle\\
&=\zeta\blacktriangleright_{\{a,b\}}\{a,c\}.
\end{align*}
Similarly, we also have $\{\zeta\blacktriangleright_{b}c,a\}=\zeta\blacktriangleright_{\{b,a\}}\{c,a\}$.\\
Therefore, the bracket $\{~ ,~\}$ is a bi-affine morphism. Now, we check the affine anti-symmetry property. For all a,b $\in \mathcal{A}$, we have
$$
\langle\{a,b\},\{a,a\},\{b,a\}\rangle=\langle ab,ba,\alpha(b),\alpha(a),ba,ab,\alpha(a)\rangle=\alpha(b)=\{b,b\}.
$$
Observed that
\begin{align*}
\{\alpha(a),\{b,c\}\}&=\{\alpha(a),\langle bc,cb,\alpha(c)\rangle\}\\
&=\langle \alpha(a)\langle bc,cb,\alpha(c)\rangle,\langle bc,cb,\alpha(c)\rangle \alpha(a),\alpha\langle bc,cb,\alpha(c)\rangle\rangle\\
&=\langle \alpha(a)(bc),\alpha(a)(cb),\alpha(ac),(bc)\alpha(a),(cb)\alpha(a),\alpha(ca),\alpha(bc),\alpha(cb),\alpha^{2}(c)\rangle.
\end{align*}
Similarly to the above, we have
\begin{align*}
&\{\alpha(b),\{c,a\}\}=\langle \alpha(b)(ca),\alpha(b)(ac),\alpha(ba),(ca)\alpha(b),(ac)\alpha(b),\alpha(ab),\alpha(ca),\alpha(ac),\alpha^{2}(a)\rangle.\\
&\{\alpha(c),\{a,b\}\}=\langle \alpha(c)(ab),\alpha(c)(ba),\alpha(cb),(ab)\alpha(c),(ba)\alpha(c),\alpha(bc),\alpha(ab),\alpha(ba),\alpha^{2}(b)\rangle.\\
&\{\alpha(a),\{a,a\}\}=\{\alpha(a),\langle aa,aa,\alpha(a)\rangle\}=\{\alpha(a),\alpha(a)\}=\alpha^{2}(a).\\
&\{\alpha(b),\{b,b\}\}=\alpha^{2}(b).\\
&\{\alpha(c),\{c,c\}\}=\alpha^{2}(c).
\end{align*}
Now, we have 
\begin{align*}
&\langle\{\alpha(a),\{b,a\}\},\{\alpha(a),\{a,a\}\},\{\alpha(b),\{c,a\}\},\{\alpha(b),\{b,b\}\},\{\alpha(c),\{a,b\}\}\rangle\\
&=\langle \alpha(a)(bc),\alpha(a)(cb),\alpha(ac),(bc)\alpha(a),(cb)\alpha(a),\alpha(ca),\alpha(bc),\alpha(cb),\alpha^{2}(c),\alpha^{2}(a),\\
&~~~~~~~~~\alpha(b)(ca),\alpha(b)(ac),\alpha(ba),(ca)\alpha(b),(ac)\alpha(b),\alpha(ab),\alpha(ca),\alpha(ac),\alpha^{2}(a),\alpha^{2}(b),\\
&~~~~~~~~~\alpha(c)(ab),\alpha(c)(ba),\alpha(cb),(ab)\alpha(c),(ba)\alpha(c),\alpha(bc),\alpha(ab),\alpha(ba),\alpha^{2}(b)\rangle\\
&=\alpha^{2}(c)\\
&=\{\alpha(c),\{c,c\}\}.
\end{align*}
Hence, the triple ($\mathcal{A},\{~,~\}$) is a Hom-Lie affgebra. 
\end{proof}
\begin{Def}
A left Hom-pre-Lie affgebra is an affine space $\mathcal{A}$ together with the bi-affine map $\bullet:\mathcal{A}\times\mathcal{A}\rightarrow \mathcal{A}$ and an unary affine morphism $\alpha:\mathcal{A}\rightarrow \mathcal{A}$, such that, for all $a,b,c\in \mathcal{A},$
\begin{equation}\label{eq2}
(a\bullet b)\bullet\alpha(c)=\langle \alpha(a)\bullet (b\bullet c),\alpha(b)\bullet(a\bullet c),(b\bullet a)\bullet\alpha(c)\rangle.
\end{equation}
Similarly, a right Hom-pre-Lie affgebra is an affine space with bi-affine binary operation $\bullet$ and unary affine morphism $\alpha$, satisfying the following condition
\begin{equation}\label{eq4}
\alpha(a)\bullet(b\bullet c)=\langle(a\bullet b)\bullet\alpha(c),(a\bullet c)\bullet \alpha(b),\alpha(a)\bullet(c\bullet b)\rangle.
\end{equation}
If the unary map $\alpha$ is multiplicative, i.e; $\alpha(a\bullet b)=\alpha(a)\bullet\alpha(b),$ for all $a,b\in \mathcal{A}$, then it is called multiplicative left(right) Hom-pre-Lie affgebra. Throughout this paper by Hom-pre-Lie affgebra we also means multiplicative Hom-pre-Lie affgebra.
\end{Def}
\begin{Rem}
It is worth noting that when these defining condition (\ref{eq2}) and (\ref{eq4}) are expressed in terms of addition $a+b=\langle a,o,b\rangle$, they become identical to the conventional definition of Hom-pre-Lie algebras.
\end{Rem}
\begin{Prop}
Let $(\mathcal{A},\bullet,\alpha)$ be a right (or left) Hom-pre-Lie affgebra. Then $\mathcal{A}$ is a Hom-Lie affgebra with the same unary affine morphism $\alpha$ and the bracket
$$
\{a,b\}=\langle a\bullet b,b\bullet a,\alpha(a)\rangle,
$$
for all $a,b\in \mathcal{A}.$
\end{Prop}
\begin{proof}
Similar to the proof of the Proposition (\ref{Prop1}), it is a routine calculation to show that $\alpha$ are affine anti-symmetry and multiplicative. Therefore, we need to only check affine Hom-Jacobi identity holds in the case of right Hom-pre Lie affgebra. For any $a,b,c\in \mathcal{A}$, we have
\begin{align*}
&\{\alpha(a),\{b,c\}\}\\
&=\{\alpha(a),\langle b\bullet c,c\bullet b,\alpha(c)\rangle\}\\
&=\langle \alpha(a)\bullet(b\bullet c),\alpha(a)\bullet(c\bullet b),\alpha(a)\bullet\alpha(c),(b\bullet c)\bullet\alpha(a),(c\bullet b)\bullet \alpha(a),\alpha(c)\bullet\alpha(a),\alpha\langle b\bullet c,c\bullet b,\alpha(c)\rangle\\
&=\langle (a\bullet b)\bullet\alpha(c),(a\bullet c)\bullet\alpha(b),\alpha(a\bullet c),(b\bullet c)\bullet\alpha(a),(c\bullet b)\bullet \alpha(a),\alpha(c\bullet a),\alpha(b\bullet c),\alpha(c\bullet b),\alpha^{2}(c)\rangle.
\end{align*}
Similarly, using condition (\ref{eq4}) and $\alpha(a\bullet b)=\alpha(a)\bullet\alpha(b)$, we have
\begin{align*}
\{\alpha(b),\{c,a\}\}&=\langle(b\bullet c)\bullet\alpha(a),(b\bullet a)\bullet\alpha(c),\alpha(b\bullet a),(c\bullet a)\bullet\alpha(b),(a\bullet c)\bullet\alpha(b),\alpha(a\bullet b),\alpha(c\bullet a),\alpha(a\bullet c),\alpha^{2}(a)\rangle,\\
\{\alpha(c),\{a,b\}\}&=\langle(c\bullet a)\bullet\alpha(b),(c\bullet b)\bullet\alpha(a),\alpha(c\bullet b),(a\bullet b)\bullet\alpha(c),(b\bullet a)\bullet\alpha(c),\alpha(b\bullet c),\alpha(a\bullet b),\alpha(b\bullet a),\alpha^{2}(b)\rangle,\\
\{\alpha(a),\{a,a\}\}&=\{\alpha(a),\langle a\bullet a,a\bullet a,\alpha(a)\rangle\}=\alpha^{2}(a),\\
\{\alpha(b),\{b,b\}\}&=\alpha^{2}(b),\\
\{\alpha(c),\{c,c\}\}&=\alpha^{2}(c).
\end{align*}
In abelian heap permuting odd position to odd and even to even, and applying the mal'tcev property, we have
\begin{align*}
&\langle\{\alpha(a),\{b,c\}\},\{\alpha(a),\{a,a\}\},\{\alpha(b),\{c,a\}\},\{\alpha(b),\{b,b\}\},\{\alpha(c),\{a,b\}\}\rangle\\
&=\langle (a\bullet b)\bullet\alpha(c),(a\bullet c)\bullet\alpha(b),\alpha(a\bullet c),(b\bullet c)\bullet\alpha(a),(c\bullet b)\bullet \alpha(a),\alpha(c\bullet a),\alpha(b\bullet c),\alpha(c\bullet b),\alpha^{2}(c),\alpha^{2}(a),\\
&~~~~~~~(b\bullet c)\bullet\alpha(a),(b\bullet a)\bullet\alpha(c),\alpha(b\bullet a),(c\bullet a)\bullet\alpha(b),(a\bullet c)\bullet\alpha(b),\alpha(a\bullet b),\alpha(c\bullet a),\alpha(a\bullet c),\alpha^{2}(a),\alpha^{2}(b),\\
&~~~~~~~~(c\bullet a)\bullet\alpha(b),(c\bullet b)\bullet\alpha(a),\alpha(c\bullet b),(a\bullet b)\bullet\alpha(c),(b\bullet a)\bullet\alpha(c),\alpha(b\bullet c),\alpha(a\bullet b),\alpha(b\bullet a),\alpha^{2}(b) \rangle\\
&=\alpha^{2}(c)\\
&=\lbrace \alpha(c),\{c,c\}\}.
\end{align*}
Therefore, ($\mathcal{A},\{~,~\},\alpha$) is a Hom-Lie affgebra. 
\end{proof}
\begin{Prop}\label{Prop3.7}
Let $(\mathcal{A},\mu)$ be an associative affgebra and $\alpha: \mathcal{A}\rightarrow \mathcal{A}$ be an affgebra morphism. Then the triple $\mathcal{A}_{\alpha}=(\mathcal{A},\mu_{\alpha}=\alpha\circ\mu,\alpha)$ is a Hom-associative affgebra.
\end{Prop}
\begin{proof}
Let $\mathcal{A}$ be an associative affgebra. To prove, $\mathcal{A}_{\alpha}$ is a Hom-associative affgebra, we only need to prove that $\mu_{\alpha}$ is bi-affine multiplication and satisfy Hom-associativity property.
For all $a,b,c,d\in \mathcal{A_{\alpha}}$ and any $\zeta\in \mathbb{K}$,
\begin{align*}
\mu_{\alpha}(a,\langle b,c,d\rangle)&=\alpha\circ\mu (a,\langle b,c,d\rangle)\\
&=\alpha\langle \mu(a,b),\mu(a,c),\mu(a,d)\rangle\\
&=\langle \alpha (\mu(a,b)),\alpha(\mu(a,c)),\alpha(\mu(a,d))\rangle\\
&=\langle \alpha \circ\mu(a,b),\alpha\circ\mu(a,c),\alpha\circ\mu(a,d)\rangle\\
&=\langle \mu_{\alpha}(a,b),\mu_{\alpha}(a,c),\mu_{\alpha}(a,d)\rangle.
\end{align*}
Therefore, $\mu_{\alpha}(a,\langle b,c,d\rangle)=\langle \mu_{\alpha}(a,b),\mu_{\alpha}(a,c),\mu_{\alpha}(a,d)\rangle$. Similarly, $\mu_{\alpha}(\langle b,c,d\rangle,a)=\langle \mu_{\alpha}(b,a),\mu_{\alpha}(c,a),\mu_{\alpha}(d,a)\rangle.$
We have 
\begin{align*}
\mu_{\alpha}(a,\zeta\blacktriangleright_{b}c)&=\alpha\circ\mu(a,\zeta\blacktriangleright_{b}c)\\
&=\alpha(\zeta\blacktriangleright_{\mu(a,b)}\mu(a,c))\\
&=\zeta\blacktriangleright_{\alpha(\mu(a,b))}\alpha(\mu(a,c))\\
&=\zeta\blacktriangleright_{\alpha\circ\mu(a,b)}\alpha\circ\mu(a,c)\\
&=\zeta\blacktriangleright_{\mu_{\alpha}(a,b)}\mu_{\alpha}(a,c).
\end{align*}
Therefore, we have $\mu_{\alpha}(a,\zeta\blacktriangleright_{b}c)=\zeta\blacktriangleright_{\mu_{\alpha}(a,b)}\mu_{\alpha}(a,c)$. Similarly, we can prove that, $\mu_{\alpha}(\zeta\blacktriangleright_{b}c,a)=\zeta\blacktriangleright_{\mu_{\alpha}(b,a)}\mu_{\alpha}(c,a)$. Hence, the multiplication $\mu_{\alpha}$ is bi-affine. It is a routine work to check that $\alpha$ is multiplicative. Now, for all $a,b,c\in \mathcal{A_{\alpha}}$, we have
\begin{align*}
\mu_{\alpha}(\alpha(a),\mu_{\alpha}(b,c))&=\mu_{\alpha}(\alpha(a),\alpha\circ\mu(b,c))\\
&=\mu_{\alpha}(\alpha(a),\mu(\alpha(b),\alpha(c)))\\
&=\mu_{\alpha}(\mu(\alpha(a),\alpha(b)),\alpha(c))\\
&=\mu_{\alpha}(\alpha\circ\mu(a,b),\alpha(c))\\
&=\mu_{\alpha}(\mu_{\alpha}(a,b),\alpha(c)).
\end{align*}
Therefore, $\mathcal{A}_{\mu_{\alpha}}$ is a Hom-associativity algebra.
\end{proof}
\begin{Prop}
Let $(\mathcal{A},\{~,~\})$ is a Lie affgebra and $\alpha:\mathcal{A}\rightarrow \mathcal{A}$ is a Lie affgebra morphism, then $(\mathcal{A_{\alpha}},\{~,~\}_{\alpha}=\alpha\circ\{~,~\},\alpha)$ is a Hom-Lie affgebra.
\begin{proof}
The proof is similar to the Proposition (\ref{Prop3.7}).
\end{proof}
\end{Prop}
\begin{Prop}
Let $(\mathcal{A},\bullet)$ is a pre-Lie affgebra and $\alpha:\mathcal{A}\rightarrow\mathcal{A}$ is pre-Lie affgebra morphism. Now we define new product $a\bullet_{\alpha}b=\alpha(a)\bullet\alpha(b)$, for all $a,b\in \mathcal{A}$. Then $(\mathcal{A}_{\alpha},\bullet_{\alpha},\alpha)$ is a Hom-pre-Lie affgebra.
\end{Prop}
From the above discussion, we have the following diagram linking the varies affgebras as follows

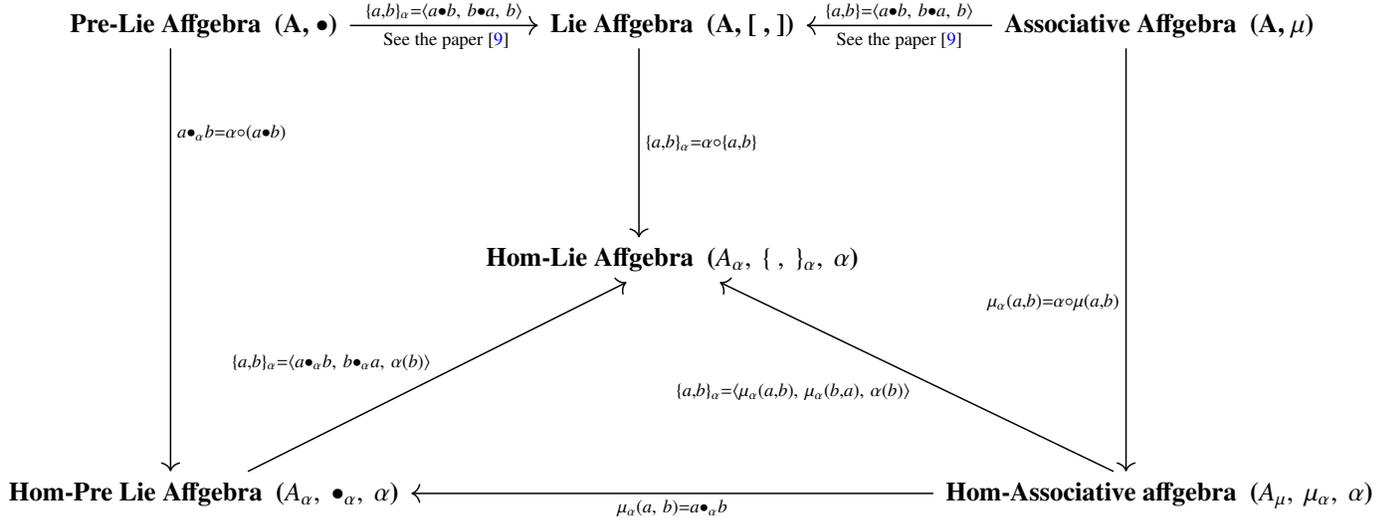
\[\begin{tikzcd}
	{\textbf{Pre-Lie Affgebra~~(A, $\bullet$)}} & {\textbf{Lie Affgebra~~(A,~[~,~])}} & {\textbf{Associative Affgebra~~(A,~$\mu$)}} \\
	\\
	\\
	\\
	& {\textbf{Hom-Lie Affgebra~~($A_{\alpha},~\{~,~\}_{\alpha},~\alpha$)}} \\
	\\
	\\
	\\
	{\textbf{Hom-Pre Lie Affgebra~~($A_{\alpha},~ \bullet_{\alpha}, ~\alpha$)}} && {\textbf{Hom-Associative affgebra~~($A_{\mu},~\mu_{\alpha},~\alpha$)}}
	\arrow["{\{a,b\}_{\alpha}=\langle a\bullet b,~b\bullet a,~b \rangle}"above,"{\text{ See the paper \cite{Brz1}}}"below, from=1-1, to=1-2]
	\arrow["{{{a\bullet_{\alpha}b=\alpha\circ(a\bullet b)}}}"{pos=0.2}, shift right=5, from=1-1, to=9-1]
	\arrow["{{{\{a,b\}_{\alpha}=\alpha\circ\{a,b\}}}}", shift right=5, from=1-2, to=5-2]
	\arrow["{\{a,b\}=\langle a\bullet b,~b\bullet a,~b \rangle}"above,"{\text{See the paper \cite{Brz1}}}"below, from=1-3, to=1-2]
	\arrow["{{{\mu_{\alpha}(a,b)=\alpha\circ\mu(a,b)}}}"'{pos=0.6}, shift right=5, from=1-3, to=9-3]
	\arrow["{{{\{a,b\}_{\alpha}=\langle a\bullet_{\alpha}b,~b\bullet_{\alpha}a,~\alpha(b)\rangle}}}"{pos=0.5}, from=9-1, to=5-2]
	\arrow["{{{\{a,b\}_{\alpha}=\langle\mu_{\alpha}(a,b),~\mu_{\alpha}(b,a),~\alpha(b)\rangle}}}", from=9-3, to=5-2]
	\arrow["{{\mu_{\alpha}(a,~b)=a\bullet_{\alpha}b}}"below, from=9-3, to=9-1]
\end{tikzcd}\]

\section{Hom-Affgebra vs Hom-Algebra}
In this section, we illustrate how any Hom-associative affgebra can be retracted to an Hom-associative algebra, any Hom-Lie affgebra can be retracted to a Hom-Lie algebra.
\begin{Thm}\label{Thm10}
Let $(\mathcal{A},\{~,~\},\alpha)$ be a Hom-Lie affgebra. For any $o\in\mathcal{A}$ and $\alpha(o)=o$, then the fibre space $T_{o}\mathcal{A}$ together with unary map $\overrightarrow{\alpha}_{o}(a)=\alpha(a)$ and the bracket
\begin{equation}\label{eq20}
[a,b]_{o}=\langle \{a,b\},\{a,o\},\{o,o\},\{o,b\},o\rangle=\{a,b\}-\{a,o\}+\{o,o\}-\{o,b\},
\end{equation}
is a Hom-Lie algebra. Moreover, any two Hom-Lie $\mathbb{K}$-algebras $(T_{o}\mathcal{A},[~,~]_{o},\overrightarrow{\alpha}_{o})$ and $(T_{e}\mathcal{A},[~,~]_{e},\overrightarrow{\alpha}_{e})$ are mutually isomorphic.
\end{Thm}
\begin{proof}
The operation $[~,~]_{o}$ is a bi-linearisation of the bi-affine bracket $\{~,~\}$ through repeated application of the linearisation Formula (\ref{eq1a}) with respect to each argument. Consequently, $[~,~]_{o}$ is a bilinear map on $T_{o}\mathcal{A}$. Similarly from the Remark (\ref{Rem1}), the map $\overrightarrow{\alpha}_{o}$ is a linearisation of $\alpha$ on $T_{o}\mathcal{A}$.\\
For all $a,b,c\in T_{o}\mathcal{A}$, using the addition and subtraction in $T_{o}\mathcal{A}$, the affine anti-symmetry condition and the affine Hom-jacobi identity for the bracket $\{~,~\}$ can be expresses as follows 
\begin{equation}\label{eq6}
\{a,b\}-\{a,a\}+\{b,a\}=\{b,b\},
\end{equation}
and
\begin{equation}\label{eq7}
\{\alpha(a),\{b,c\}\}-\{\alpha(a),\{a,a\}\}+\{\alpha(b),\{c,a\}\}-\{\alpha(b),\{b,b\}\}+\{\alpha(c),\{a,b\}\}=\{\alpha(c),\{c,c\}\}.
\end{equation}

The Equation (\ref{eq6}) implies that $[a,a]_{o}=o$ and $[a,b]_{o}=-[b,a]_{o}$.  That is, $[~,~]_{o}$ is satisfy anti-symmetry condition.
Now to verify the Hom-Jacobi identity, observe that
\begin{align*}
&[\alpha(a),[b,c]_{o}]_{o}\\
&=[\alpha(a),\{b,c\}-\{b,o\}+\{o,o\}-\{o,c\}]_{o}\\
&=[\alpha(a),\{b,c\}]_{o}-[\alpha(a),\{b,o\}]_{o}+[\alpha(a),\{o,o\}]_{o}-[\alpha(a),\{o,c\}]_{o}\\
&=\big{(}\{\alpha(a),\{b,c\}\}-\{\alpha(a),o\}+\{o,o\}-\{o,\{b,c\}\}\big{)}-\big{(}\{\alpha(a),\{b,o\}\}-\{\alpha(a),o\}+\{o,o\}-\{o,\{b,o\}\}\big{)}\\
&~~+\big{(}\{\alpha(a),\{o,o\}\}-\{\alpha(a),o\}+\{o,o\}-\{o,\{o,o\}\}\big{)}-\big{(}\{\alpha(a),\{o,c\}\}-\{\alpha(a),o\}+\{o,o\}-\{o,[\{o,c\}\}\big{)}\\
&=\{\alpha(a),\{b,c\}\}-\{o,\{b,c\}\}-\{\alpha(a),\{b,o\}\}+\{o,\{b,o\}\}+\{\alpha(a),\{o,o\}\}-\{o,\{o,o\}\}-\{\alpha(a),\{o,c\}\}+\{o,\{o,c\}\}.
\end{align*}
Similarly, we have
\begin{align*}
[\alpha(b),[c,a]_{o}]_{o}&=\{\alpha(b),\{c,a\}\}-\{o,\{c,a\}\}-\{\alpha(b),\{c,o\}\}+\{o,\{c,o\}\}+\{\alpha(b),\{o,o\}\}-\{o,\{o,o\}\}-\{\alpha(b),\{o,a\}\}+\{o,\{o,a\}\};\\
[\alpha(c),[a,b]_{o}]_{o}&=\{\alpha(c),\{a,b\}\}-\{o,\{a,b\}\}-\{\alpha(c),\{a,o\}\}+\{o,\{a,o\}\}+\{\alpha(c),\{o,o\}\}-\{o,\{o,o\}\}-\{\alpha(c),\{o,b\}\}+\{o,\{o,b\}\}.
\end{align*}
Applying the Equation (\ref{eq7}) and $\alpha(o)=o$, we have
\begin{align*}
&[\alpha(a),[b,c]_{o}]_{o}+[\alpha(b),[c,a]_{o}]_{o}+[\alpha(c),[a,b]_{o}]_{o}\\
&=\big{(}\{\alpha(a),\{b,c\}\}-\{o,\{b,c\}\}-\{\alpha(a),\{b,o\}\}+\{o,\{b,o\}\}+\{\alpha(a),\{o,o\}\}-\{o,\{o,o\}\}-\{\alpha(a),\{o,c\}\}+\{o,\{o,c\}\}\big{)}\\
&~~+\big{(}\{\alpha(b),\{c,a\}\}-\{o,\{c,a\}\}-\{\alpha(b),\{c,o\}\}+\{o,\{c,o\}\}+\{\alpha(b),\{o,o\}\}-\{o,\{o,o\}\}-\{\alpha(b),\{o,a\}\}+\{o,\{o,a\}\}\big{)}\\
&~~+\big{(}\{\alpha(c),\{a,b\}\}-\{o,\{a,b\}\}-\{\alpha(c),\{a,o\}\}+\{o,\{a,o\}\}+\{\alpha(c),\{o,o\}\}-\{o,\{o,o\}\}-\{\alpha(c),\{o,b\}\}+\{o,\{o,b\}\}\big{)}\\
&=o.
\end{align*}
To prove the last part, we use the translation isomorphism $\tau_{o}^{e}(a)=a-o$ (In the sense of binary operation in $T_{e}\mathcal{A}$) describe in Remark (\ref{Rem1}). Now only need to prove that $\tau_{o}^{e}$ is Hom-Lie algebra morphism,
\begin{align*}
[\tau_{o}^{e}(a),\tau^{e}_{o}(b)]_{e}&=[a-o,b-o]_{e}\\
&=[a,b]_{e}-[a,o]_{e}-[o,b]_{e}+[o,o]_{e}\\
&=\big{(}\{a,b\}-\{a,e\}+\{e,e\}-\{e,b\}\big{)}-\big{(}\{a,o\}-\{a,e\}+\{e,e\}-\{e,o\}\big{)}\\
&~~~-\big{(}\{o,b\}-\{o,e\}+\{e,e\}-\{e,b\}\big{)}+\big{(}\{o,o\}-\{o,e\}+\{e,e\}-\{e,o\}\big{)}\\
&=\{a,b\}-\{a,o\}+\{o,o\}-\{o,b\}\\
&=[a,b]_{o}\\
&=[a,b]_{o}-o~~~~~~~~\text{\{In $T_{o}\mathcal{A}$, $o$ is look like identity.\}}\\
&=\tau_{o}^{e}([a,b]_{o}).
\end{align*}
We also have
\begin{align*}
\overrightarrow{\alpha}_{e}\circ\tau_{o}^{e}(a)&=\overrightarrow{\alpha}_{e}(a-o)\\
&=\alpha(a-o)-\alpha(e)\\
&=\alpha(a-o)\\
&=\overrightarrow{\alpha}_{o}(a-o)\\
&=\overrightarrow{\alpha}_{o}(a)-\overrightarrow{\alpha}_{o}(o)\\
&=\overrightarrow{\alpha}_{o}(a)-o\\
&=\tau_{o}^{e}(\overrightarrow{\alpha}_{o}(a)).
\end{align*}
\end{proof}
\begin{Thm}
Let $(\mathcal{A},\bullet,\alpha)$ be a Hom-associative affgebra. For all $o\in \mathcal{A}$ and $\alpha(o)=o$, then the fibre space $T_{o}\mathcal{A}$ together with unary map $\overrightarrow{\alpha}_{o}(a)=\alpha(a)$ and, for all $a,b\in T_{o}\mathcal{A}$, the multiplication 
\[
a\bullet_{o}b=a\bullet b-a\bullet o+o^{2}-o\bullet b+o,
\]
is a Hom-associative algebra. Furthermore, any two Hom-associative algebras $(T_{o}\mathcal{A},\bullet_{o},\alpha)$ and $(T_{e}\mathcal{A},\bullet_{e},\alpha)$ are mutually isomorphic.
\end{Thm}
\begin{proof}
The multiplication $\bullet_{o}$ is a bi-linearisation of the bi-affine multiplication $\bullet$ through repeated application of the linearisation Formula (\ref{eq1a}) with respect to each argument. Consequently, $\bullet_{o}$ is a bilinear map on $\mathcal{A}_{o}$. Similarly from the Remark (\ref{Rem1}) the map $\overrightarrow{\alpha}_{o}$ is a linearisation of $\alpha$ map on $\mathcal{A}_{o}$. Now, we only need to prove the Hom-associative property. For all $a,b,c\in \mathcal{A_{\alpha}}$, we have
\begin{align*}
&\alpha(a)\bullet_{o}(b\bullet_{o}c)\\
&=\alpha(a)\bullet_{o}(b\bullet c-b\bullet o+o^{2}-o\bullet c+o)\\
&=\alpha(a)\bullet(b\bullet c-b\bullet o+o^{2}-o\bullet c+o)-\alpha(a)\bullet o+o^{2}-o\bullet(b\bullet c-b\bullet o+o^{2}-o\bullet c+o)+o\\
&=\alpha(a)\bullet(b\bullet c)-\alpha(a)\bullet(b\bullet o)+\alpha(a)\bullet o^{2}-\alpha(a)\bullet(o\bullet c)+\alpha(a)\bullet o-\alpha(a)\bullet o+o^{2}-o\bullet(b\bullet c)\\
&~~~~+o\bullet(b\bullet o)-o\bullet(o^{2})+o\bullet(o\bullet c)+(o\bullet o)+o\\
&=\{\alpha(a)\bullet(b\bullet c)-\alpha(a)\bullet(o\bullet c)+o\bullet(o\bullet c)-o\bullet(b\bullet c)\}-\{\alpha(a)\bullet(b\bullet o)-\alpha(a)\bullet(o\bullet o)+o\bullet(o\bullet o)\\
&~~~-o\bullet(b\bullet o)+o\bullet o\}+o^{2}+o\bullet\alpha( c)-o\bullet \alpha( c)\\
&=\{(a\bullet b)\bullet \alpha(c)-(a\bullet o)\bullet\alpha( c)+(o\bullet o)\bullet \alpha( c)-(o\bullet b)\bullet\alpha( c)\}+o\bullet\alpha(c)-\{(a\bullet b)\bullet \alpha( o)-(a\bullet o)\bullet\alpha( o)\\
&~~~~+(o\bullet o)\bullet\alpha( o)-(o\bullet b)\bullet\alpha( o)+o\bullet\alpha( o)\}+o^{2}+-o\bullet\alpha( c)\\
&=(a\bullet b-a\bullet o+o^{2}-o\bullet b+o)\bullet\alpha(c)-(a\bullet b-a\bullet o+o^{2}-o\bullet b+o)\bullet\alpha( o)+o^{2}-o\bullet\alpha(c)\\
&=(a\bullet b-a\bullet o+o^{2}-o\bullet b+o)\bullet\alpha(c)-(a\bullet b-a\bullet o+o^{2}-o\bullet b+o)\bullet o+o^{2}-o\bullet\alpha(c)\\
&=(a\bullet_{o}b)\bullet_{o}\alpha(c).
\end{align*}
The last part of the Theorem is similarly to the prove of the Theorem (\ref{Thm10}). 
\end{proof}
\begin{Thm}
Let $(\mathcal{L},[-,-],\alpha)$ be a Hom-Lie algebra. Then for all $r\in \mathcal{L}$, and with the same unary map $\alpha$ along with the affine space operation
$$
\langle a,b,c\rangle=a-b+c,~~~~~~~~\xi\blacktriangleright_{a}b=\xi b+(1-\xi )a,~~~~~~~\text{for all $a,b,c\in \mathcal{L}$, $\xi\in \mathbb{K}$},
$$
and the affine bracket
$$
\{a,b\}=[a,b]+\alpha(b)+r, ~~~~\text{for all $a,b\in \mathcal{L}$},
$$
$\mathcal{L}$ is a Hom-Lie affgebra.
\end{Thm}
\begin{proof}
Let $(\mathcal{L},[-,-],\alpha)$ be a Hom-Lie algebra. By the definition that the bracket $[-,-]$ and unary map $\alpha$  are bi-linear and linear respectively. It is well know result that, every linear map is an affine map. So, the affine bracket $\{-,-\}$ is a bi-affine map as this is sum of linear and constant part. Now to prove affine anti-symmetry, for any $a,b\in \mathcal{L}$, we have 
\begin{align*}
\langle\{a,b\},\{a,a\},\{b,a\}\rangle&=\{a,b\}-\{a,a\}+\{b,a\}\\
&=[a,b]+\alpha(b)+r-[a,a]-\alpha(a)-r+[b,a]+\alpha(a)+r\\
&=\alpha(b)+r\\
&=\{b,b\}.
\end{align*}
Hence, the affine anti-symmetry property holds.
To show the affine Hom-Jacodi identity of the affine bracket $\{-,-\}$, for all $a,b,c\in \mathcal{L}$, we get
\begin{align*}
&~~~~\langle\{\alpha(a),\{b,c\}\},\{\alpha( a),\{a,a\}\},\{\alpha(b),\{c,a\}\},\{\alpha(b),\{b,b\}\},\{\alpha(c),\{a,b\}\}\rangle\\
&=\{\alpha(a),\{b,c\}\}-\{\alpha( a),\{a,a\}\}+\{\alpha(b),\{c,a\}\}-\{\alpha(b),\{b,b\}\}+\{\alpha(c),\{a,b\}\}\\
&=[\alpha(a),[b,c]]+[\alpha(a),\alpha(c)]+[\alpha(a),r]+\alpha[b,c]-\alpha^{2}(c)+\alpha(r)+r-[\alpha(a),r]-\alpha^{2}(a)-\alpha(r)+r\\
&~~~~~[\alpha(b),[c,a]]+[\alpha(b),\alpha(a)]+[\alpha(b),r]+\alpha[c,a]+\alpha^{2}(a)+\alpha(r)+r-[\alpha(b),r]-\alpha^{2}(b)-\alpha(r)-r\\
&~~~~~[\alpha(c),[a,b]]+[\alpha(c),\alpha(b)]+[\alpha(c),r]+\alpha[a,b]+\alpha^{2}(b)+\alpha(r)+r\\
&=[\alpha(c),r]+\alpha^{2}(c)+\alpha(r)+r\\
&=\{\alpha(c),\{c,c\}\}.
\end{align*}
Thus, the affine Hom-Jacobi identity holds.
\end{proof}

The next theorem is more general then above theorem. Which show a close connection between generalized derivation of Hom-Lie algebra and Hom-Lie affgebras.

\begin{Thm}\label{Thm6.4}
Let $\mathcal{L} = (\mathcal{L}, [-,-], \alpha)$ be a Hom-Lie algebra, and let $\kappa$ and $\lambda$ be linear endomorphisms of $\mathcal{L}$ satisfying $\lambda\alpha = \alpha\lambda$ and $\kappa\alpha = \alpha\kappa$, such that for all $a,b \in \mathcal{L}$,
\begin{equation}\label{eq17}
\lambda([a,b]) = [\lambda(a), \alpha(b)] - [\alpha(a), \kappa(b)] + [\alpha(a), \lambda(b)].
\end{equation}
Then, for any $r \in \mathcal{L}$, the affine space with the same unary map $\alpha$, equipped with the operations
\[
\langle a, b, c \rangle = a - b + c, \qquad 
\xi \blacktriangleright_{a} b = \xi b + (1 - \xi)a, \qquad 
\text{for all } a,b,c \in \mathcal{L},~ \xi \in \mathbb{K},
\]
and the affine bracket
\begin{equation}\label{eq18}
\{a,b\} = [a,b] + \kappa(a) + \lambda(b - a) + r, \qquad \text{for all } a,b \in \mathcal{L},
\end{equation}
forms a Hom-Lie affgebra, denoted by $\mathcal{A}(\mathcal{L}; \alpha, \kappa, \lambda, r)$.

Moreover, for every $o \in \mathcal{A}$ satisfying $\alpha(o) = o$, we have
\[
T_{o}\mathcal{A}(\mathcal{L}; \alpha, \kappa, \lambda, r) \cong \mathcal{L}.
\]
Conversely, for any Hom-Lie affgebra $(\mathcal{A}, \{\, , \,\}, \alpha)$ and any $o \in \mathcal{A}$ such that $\alpha(o) = o$, there exist linear maps $\kappa, \lambda$ and an element $r$ necessarily satisfying \eqref{eq17}, such that
\[
\mathcal{A} = \mathcal{A}(T_{o}\mathcal{A}, \alpha, \kappa, \lambda, r).
\]
\end{Thm}

\begin{proof}
Let $(\mathcal{L},[-,-],\alpha)$ be a Hom-Lie algebra. Now to prove affine anti-symmetry, take any $a,b\in \mathcal{L}$, we have 
\begin{align*}
\langle\{a,b\},\{a,a\},\{b,a\}&=\{a,b\}-\{a,a\}+\{b,a\}\\
&=[a,b]+\kappa(a)+\lambda(b-a)+r-\kappa(a)-r+[b,a]+\kappa(b)+\lambda(a-b)+r\\
&=\kappa(b)+r\\
&=\{b,b\}.
\end{align*}
Hence, the affine anti-symmetry property holds.
To show the affine Hom-Jacodi identity of the affine bracket $\{-,-\}$, for all $a,b,c\in \mathcal{L}$, let us first note that
\begin{align*}
\{\alpha(a),\{b,c\}\}-\{\alpha(a),\{a,a\}\}=&[\alpha(a),[b,c]]+[\alpha(a),\kappa(b)]+[\alpha(a),\lambda(c)]-[\alpha(a),\lambda(b)]\\
&+\lambda[b,c]+\lambda\circ\kappa(b-a)+\lambda^{2}(c-b)-[\alpha(a),\kappa(a)]
\end{align*}

Now, the affine bracket defined by the formula (\ref{eq18}) satisfies affine Hom-Jacobi identity if and only if
\begin{align*}
0&=\langle\{\alpha(a),\{b,c\}\},\{\alpha( a),\{a,a\}\},\{\alpha(b),\{c,a\}\},\{\alpha(b),\{b,b\}\},\{\alpha(c),\{a,b\}\}\rangle-\{\alpha(c),\{c,c\}\}\\
&=\{\alpha(a),\{b,c\}\}-\{\alpha( a),\{a,a\}\}+\{\alpha(b),\{c,a\}\}-\{\alpha(b),\{b,b\}\}+\{\alpha(c),\{a,b\}\}-\{\alpha(c),\{c,c\}\}\\
&=[\alpha(a),[b,c]]+[\alpha(a),\kappa(b)]+[\alpha(a),\lambda(c)]-[\alpha(a),\lambda(b)]+\lambda[b,c]-[\alpha(a),\kappa(a)]\\
&~~~+[\alpha(b),[c,a]]+[\alpha(b),\kappa(c)]+[\alpha(b),\lambda(a)]-[\alpha(b),\lambda(c)]+\lambda[c,a]-[\alpha(b),\kappa(b)]\\
&~~~+[\alpha(c),[a,b]]+[\alpha(c),\kappa(a)]+[\alpha(c),\lambda(b)]-[\alpha(c),\lambda(a)]+\lambda[a,b]-[\alpha(c),\kappa(c)]\\
&=[\alpha(a),\kappa(b)]+[\alpha(a),\lambda(c)]-[\alpha(a),\lambda(b)]+\lambda[b,c]-[\alpha(a),\kappa(a)]\\
&~~~+[\alpha(b),\kappa(c)]+[\alpha(b),\lambda(a)]-[\alpha(b),\lambda(c)]+\lambda[c,a]-[\alpha(b),\kappa(b)]\\
&~~~+[\alpha(c),\kappa(a)]+[\alpha(c),\lambda(b)]-[\alpha(c),\lambda(a)]+\lambda[a,b]-[\alpha(c),\kappa(c)].\\
\end{align*}
For a particular case, when b=c, the equality hold if and only if $[\alpha(b-a),\kappa(b-a)]=0$, that is
\begin{equation}\label{eq19}
[\alpha(a),\kappa(a)]=0,~~~~~~~~\text{for all a} \in \mathcal{L}.
\end{equation}
Next, put $c=0$ and in view of (\ref{eq19}), we get the necessity of (\ref{eq17}). Also, (\ref{eq19}) follows from (\ref{eq17}), by setting $b=a$ in (\ref{eq17}). Therefore, the bracket (\ref{eq18}) satisfies the affine Hom-Jacobi identity.\\
Proof of the part $T_{o}\mathcal{A}(\mathcal{L};\alpha,\kappa,\lambda,r)\cong\mathcal{L}$, we have already done in Theorem (\ref{Thm10}).

Conversely, to each element $o\in \mathcal{A}$ such that $\alpha(o)=o$, we assign linear maps $\kappa,\lambda:T_{o}\mathcal{A}\rightarrow T_{o}\mathcal{A}$ and an element $r\in \mathcal{A}$ which satisfying the condition (\ref{eq17}). A comparison between (\ref{eq18}) and (\ref{eq20}) suggest the following natural choices for these data
\begin{align*}
r&=\{o,o\},\\
\lambda:T_{o}\mathcal{A}\rightarrow T_{o}\mathcal{A},&~~~~~a\mapsto\{o,a\}-\{o,o\},\\
\kappa:T_{o}\mathcal{A}\rightarrow T_{o}\mathcal{A},&~~~~~a\mapsto\{a,a\}-\{o,o\}.\\
\end{align*}

Since $\{o,-\}:\mathcal{A} \to \mathcal{A}$ is an affine map (fixing the base point $o$), so $\lambda$ is linear map. Starting from the identity obtained by affine anti-symmetry, we have
\[
\kappa(a) = \{a,a\} - \{o,o\}
= (\{o,a\} - \{o,o\}) + (\{a,o\} - \{o,o\})
= \lambda(a) + (\{a,o\} - \{o,o\}).
\]
Thus, $\kappa$ is the sum of two maps, $\lambda$ (already linear) and $a \mapsto \{a,o\} - \{o,o\}$.

Since $\{-,o\}$ is affine in the first argument, the map $a \mapsto \{a,o\} - \{o,o\}$ is linear. Therefore, $\kappa$ is linear map. By using $\alpha(o)=o$, we can easily verify that $\lambda\alpha=\alpha\lambda$ and $\kappa\alpha=\alpha\kappa$.

We want to check that, with these definitions, the identity (\ref{eq18}) hold. Starting from the right-hand side and using the bracket formula (\ref{eq20})
\[
[a,b] = \{a,b\} - \{a,o\} + \{o,o\} - \{o,b\},
\]
together with
\[
\kappa(a) = \{a,a\} - \{o,o\}, \qquad 
\lambda(b-a) = \{o,b-a\} - \{o,o\}.
\]

We compute
\[
\begin{aligned}
[a,b] + \kappa(a) + \lambda(b-a) + r
&= (\{a,b\} - \{a,o\} + \{o,o\} - \{o,b\}) 
  + (\{a,a\} - \{o,o\}) 
  + (\{o,b\} - \{o,a\}) 
  + \{o,o\} \\
&= \{a,b\} + (\{a,a\} - \{a,o\} - \{o,a\}) + \{o,o\}.
\end{aligned}
\]

Now applying affine anti-symmetry with $b = o$, we get $\{a,o\} - \{a,a\} + \{o,a\} = \{o,o\}
\quad \Longrightarrow \quad
\{a,a\} - \{a,o\} - \{o,a\} = -\{o,o\}.
$ Then substituting this value, we find
\[
[a,b] + \kappa(a) + \lambda(b-a) + r = \{a,b\}.
\]
Hence, the right-hand side equals the left-hand side, and (\ref{eq18}) is identically satisfied for our choice of $r, \lambda, \kappa$. From the first part of the proof ensures that the condition (\ref{eq17}) is fulfilled.
\end{proof}
Motivating to the concept of generalized derivation of Lie algebra in the Paper \cite{Leg1}, similarly, we develop the concept of generalized derivation of Hom-Lie algebra in this section.
\begin{Def}
Let $\mathcal{L} = (\mathcal{L}, [\, , \,], \alpha)$ be a Hom-Lie algebra.  
A linear map $\lambda$ on $\mathcal{L}$ is called a \emph{generalized derivation} of $\mathcal{L}$ if there exist two linear maps $\lambda', \lambda''$ on $\mathcal{L}$ such that, for all $a,b \in \mathcal{L}$,
\begin{equation}\label{eq21}
[\lambda(a), \alpha(b)] + [\alpha(a), \lambda'(b)] = \lambda''([a,b]),
\end{equation}
and
\begin{equation}\label{eq22}
\lambda \circ \alpha = \alpha \circ \lambda, 
\qquad 
\lambda' \circ \alpha = \alpha \circ \lambda', 
\qquad 
\lambda'' \circ \alpha = \alpha \circ \lambda''.
\end{equation}
\end{Def}

We denote by $\Delta(\mathcal{L})$ the set of triples $(\lambda, \lambda', \lambda'')$ of linear endomorphisms of $\mathcal{L}$ that satisfy the condition (\ref{eq21}) and (\ref{eq22}). By slightly bending this terminology, in what follows we will refer either to $\lambda$ or the triple $(\lambda, \lambda', \lambda'')$ as a generalized derivation. We can trivially verify that, $\lambda$ is a $\alpha$-\emph{derivation} of $\mathcal{L}$ if and only if $(\lambda, \lambda, \lambda) \in \Delta(\mathcal{L})$.  

A \emph{quasi-centroid} of $\mathcal{L}$, denoted by $QC(\mathcal{L})$, is a linear map 
$\kappa$ on $\mathcal{L}$ such that $(\kappa, -\kappa, 0) \in \Delta(\mathcal{L})$. 
That is, $\kappa \in QC(\mathcal{L})$ if and only if, for all $a, b \in \mathcal{L}$,
\begin{equation}\label{eq23}
[\kappa(a), \alpha(b)] = [\alpha(a), \kappa(b)]
\quad \text{and} \quad 
(\kappa \circ \alpha)(a) = (\alpha \circ \kappa)(a).
\end{equation}

Since we assume that $\mathrm{char}\,\mathbb{K} \neq 2$, the condition (\ref{eq23}) is equivalent to (\ref{eq19}). Next, we define the \emph{centroid} of $\mathcal{L}$, denoted by $C(\mathcal{L})$, is the space of linear endomorphisms $\kappa$ such that $(0, \kappa, \kappa) \in \Delta(\mathcal{L})$, or equivalently by the anti-symmetry of the Lie bracket, $(\kappa, 0, \kappa) \in \Delta(\mathcal{L})$.  
That is, $\kappa \in C(\mathcal{L})$ if and only if, for all $a, b \in \mathcal{L}$,
\begin{equation}
\kappa([a, b]) = [\kappa(a),\alpha(b)] = [\alpha(a), \kappa(b)]~~~\text{and}~~~(\kappa\circ\alpha)(a)=(\alpha\circ\kappa)(a).
\end{equation}
Obviously $C(\mathcal{L}) \subseteq QC(\mathcal{L})$.

\begin{Lemma}\label{Lemma6.6}
Let $\mathcal{L} = (\mathcal{L}, [\, , \,], \alpha)$ be a Hom-Lie algebra, and let $\lambda$ and $\delta$ be linear maps on $\mathcal{L}$.  
Set $\kappa := \lambda - \delta$. Then the following statements are equivalent:
\begin{enumerate}
\item[(i)] $(\delta, \lambda, \lambda) \in \Delta(\mathcal{L});$
\item[(ii)] $\lambda$ and $\kappa$ satisfy \eqref{eq17}, and moreover, $\lambda \circ \alpha = \alpha \circ \lambda$ and $\kappa \circ \alpha = \alpha \circ \kappa$.
\end{enumerate}
\end{Lemma}

\begin{proof}
\textbf{(i) $\Rightarrow$ (ii)}
Assume first that $(\delta, \lambda, \lambda) \in \Delta(\mathcal{L})$.  
By definition of $\Delta(\mathcal{L})$, we have
\begin{equation*}
[\delta(a),\alpha(b)] + [\alpha(a), \lambda(b)] = \lambda([a,b]),~~~\text{and}~~~\delta\circ\alpha=\alpha\circ\delta,~~\lambda\circ\alpha=\alpha\circ\lambda \qquad \forall a,b \in \mathcal{L}.
\end{equation*}
Set $\kappa := \lambda - \delta$.  Then $\delta = \lambda - \kappa$, and substituting the value. we get
\[
[(\lambda - \kappa)(a),\alpha(b)] + [\alpha(a), \lambda(b)] = \lambda([a,b])~~~\text{and}~~~(\lambda-\kappa)\circ\alpha=\alpha\circ(\lambda-\kappa),~~\lambda\circ\alpha=\alpha\circ\lambda.
\]
Simplifying gives
\[
[\lambda(a),\alpha(b)] - [\kappa(a),\alpha(b)] + [\alpha(a), \lambda(b)] = \lambda([a,b])~~~\text{and}~~~\lambda\circ\alpha-\kappa\circ\alpha=\alpha\circ\lambda-\alpha\circ\kappa,~~\lambda\circ\alpha=\alpha\circ\lambda.
\]
Set $a=b$, then $[\kappa(a),\alpha(a)]=0$. So by equation (\ref{eq23})
\[
[\lambda(a),\alpha(b)] - [\alpha(a),\kappa(b)] + [\alpha(a), \lambda(b)] = \lambda([a,b])~~~\text{and}~~~\kappa\circ\alpha=\alpha\circ\kappa,~~\lambda\circ\alpha=\alpha\circ\lambda.
\]
Hence, $\lambda$ and $\kappa$ satisfy (\ref{eq17}) and $\kappa\circ\alpha=\alpha\circ\kappa,~~\lambda\circ\alpha=\alpha\circ\lambda$.

\textbf{(i)$\Leftarrow$ (ii)} Conversely, suppose that $\lambda$ and $\kappa$ satisfy (\ref{eq17}), and $\kappa\circ\alpha=\alpha\circ\kappa,~~\lambda\circ\alpha=\alpha\circ\lambda$. That is; $[\lambda(a),\alpha(b)] - [\alpha(a),\kappa(b)] + [\alpha(a), \lambda(b)] = \lambda([a,b])$.
Replacing $\kappa = \lambda - \delta$, we have
\[
[\lambda(a),\alpha(b)] - [\alpha(a),\lambda(b)-\delta(b)] + [\alpha(a), \lambda(b)] = \lambda([a,b]),~~~and~~~\delta\circ\alpha=\alpha\circ\delta,~~\lambda\circ\alpha
=\alpha\circ\lambda.\]
which simplifies to
\[
[\lambda(a),\alpha(b)]+[\alpha(a),\delta(b)]=
\lambda([a,b]),~~~and~~~\delta\circ\alpha=\alpha\circ\delta,~~\lambda\circ\alpha
=\alpha\circ\lambda.\\
\]
Also by anti-symmetry of Lie bracket, $[\delta(a),\alpha(b)]+[\alpha(a),\lambda(b)]=
\lambda([a,b])$. 
Thus $(\delta, \lambda, \lambda) \in \Delta(\mathcal{L})$.  Therefore, (i) and (ii) are equivalent.
\end{proof}
The assertion of Lemma (\ref{Lemma6.6}) allows one to sate Theorem (\ref{Thm6.4}) in the following equivalent form
\begin{Coro}
There is a one-to-one correspondence between Hom-Lie affgebras with Hom-Lie algebra fibre $\mathcal{L}$ and generalized derivations $(\delta, \lambda, \lambda)$ of $\mathcal{L}$, supplemented by an element of $\mathcal{L}$.
\end{Coro}

Now we define to the relationship of homomorphisms between Hom-Lie affgebras and Hom-Lie algebras for further study of Hom-Lie affgebra and algebra.

\begin{Thm}\label{Thm6.8}
A function 
$\varphi : \mathcal{A}(\mathcal{L}; \alpha, \kappa, \lambda, r) \longrightarrow \mathcal{A}(\mathcal{L}'; \alpha', \kappa', \lambda', r')$
is a homomorphism of Hom-Lie affgebras if and only if there exist a Hom-Lie algebra homomorphism 
$
\psi : \mathcal{L} \longrightarrow \mathcal{L}'$
and an element $q' \in \mathcal{L}'$ such that
\begin{subequations}\label{15abcd}
\begin{align}
q' &= \alpha'(q'), \label{15a}\\
\psi \kappa &= \kappa' \psi, \label{15b}\\
\psi \lambda &= (\mathrm{ad}_{q'} + \lambda') \psi, \label{15c}\\
\psi(r) &= r' - q' + \kappa'(q'). \label{15d}
\end{align}
\end{subequations}
\end{Thm}

\begin{proof}
From the Remark (\ref{Rem1}), a function $\varphi:\mathcal{A}(\mathcal{L};\alpha,\kappa,\lambda,r)\longrightarrow\mathcal{A}(\mathcal{L}';\alpha',\kappa',\lambda',r')$ is a homomorphism of affine spaces if and only if the function $\psi:\mathcal{L}\rightarrow\mathcal{L}'$ define as $\psi(a)=\varphi(a)-\varphi(o)$ is a linear transformation. Set $q'=\varphi(o)$.

The Hom-Lie affgebra homomorphism condition can be expressed explicitly as, for all $a,b\in \mathcal{L}$,
\begin{subequations}
\begin{align}
(\varphi\alpha)(a)&=(\alpha'\varphi)(a),\label{16a}\\
\varphi(\{a,b\})&=\{\varphi(a),\varphi(b)\}\label{16b}.
\end{align}
\end{subequations}
From the equation (\ref{16a}), reduce to 
$$
\psi(\alpha(a))+\varphi(o)=\alpha'(\psi(a)+\varphi(o)).
$$
Finally, setting $a=o$ and putting $q'=\varphi(o)$ yield the equation (\ref{15a}) as well as the fact $\psi$ satisfy the condition $\psi\alpha=\alpha'\psi$.

Similar to the prove of the Theorem (3.4) of the Paper \textbf{ \cite{And2}}, from the equation (\ref{16b}), we can prove the equation (\ref{15b}), (\ref{15c}), and (\ref{15d}) as well as the fact $\psi$ satisfy the condition $\psi[a,b]=[\psi(a),\psi(b)]$.

The converse is straightforward by simple calculation.
\end{proof}

\begin{Coro}\label{Coro6.9}
Let $\mathcal{A}=\mathcal{A}(\mathcal{L};\alpha,\kappa,\lambda,r)$ and  $\mathcal{A}'=\mathcal{A}(\mathcal{L}';\alpha',\kappa',\lambda',r')$ be two Hom-Lie affgebras with Hom-Lie algebras $\mathcal{L}$ and $\mathcal{L}'$ respectively. Then $\mathcal{A}$ is isomorphic to $\mathcal{A}'$ if and only if there exist a Hom-Lie algebra isomorphism $\Psi:\mathcal{L}\longrightarrow\mathcal{L}'$ and $q'\in \mathcal{L}'$, such that 
\begin{subequations}\label{17abcd}
\begin{align}
q'&=\alpha'(\Psi(q)),\label{17a}\\
\kappa' &=\Psi \kappa \Psi^{-1}, \label{17b}\\
\lambda' &= \Psi(\lambda-\mathrm{ad}_q)\Psi^{-1}, \label{17c}\\
r' &=\Psi( r + q-\kappa(q)). \label{17d}
\end{align}
\end{subequations}
\end{Coro}
\begin{proof}
This result follows directly from Theorem~\ref{Thm6.8}, since the affine map 
$\Phi:\mathcal{A}\rightarrow\mathcal{A}'$ is an isomorphism if and only if its linear part 
$\Psi = \Phi - \Phi(o) : \mathcal{L} \rightarrow \mathcal{L}'$ 
is an isomorphism of vector spaces. 

By rearranging equations~\eqref{15abcd}, let $q = \Psi^{-1}(q')$, where $q' = \Phi(o)$. 
Using the fact that 
\[
\Psi\,\mathrm{ad}_{q}\,\Psi^{-1}(a)
= \Psi[q, \Psi^{-1}(a)]
= \Psi[\Psi^{-1}(q'), \Psi^{-1}(a)]
= [q', a]
= \mathrm{ad}_{q'}(a),
\]
we obtain the conditions~\eqref{17abcd}.

\end{proof}
As a direct consequence of Corollary~\ref{Coro6.9}, we obtain the following criterion for the non-isomorphism of Hom-Lie affgebras.

\begin{Coro}
If two Hom-Lie affgebras have non-isomorphic Hom-Lie algebra fibres, then they are not isomorphic.
\end{Coro}

\begin{center}
 {\bf ACKNOWLEDGEMENT}
 \end{center}
 
This research was supported by the Core Research Grant (CRG) of the Anusandhan National Research Foundation (ANRF), formerly the Science and Engineering Research Board (SERB), under the Department of Science and Technology (DST), Government of India (Grant No.~CRG/2022/005332). The authors gratefully acknowledge the financial support received from the aforementioned agency.

\end{document}